\documentclass[noamsfonts,a4paper,10pt]{amsart}
\usepackage[margin=2cm]{geometry}
\usepackage[bitstream-charter,cal]{mathdesign}
\usepackage{hyperref}
\usepackage{setspace}
\theoremstyle{plain}
\newtheorem*{theorem*}{Theorem}

\title{A counterexample to Haagerup's problem on subadditive weights}
\author[I. Choi]{Ikhan Choi}
\address{Graduate School of Mathematical Sciences\\The University of Tokyo\\3-8-1 Komaba Meguro-ku Tokyo 153-8914, Japan}
\email{choi@ms.u-tokyo.ac.jp}
\author[S. Jo]{Sunghyeon Jo}
\address{QED Audit and Georgia Institute of Technology}
\email{sjo65@gatech.edu}

\begin{document}

\begin{abstract}
We construct a subadditive weight on the type I$_\infty$ factor that preserves directed suprema but is not $\sigma$-weakly lower semi-continuous.
This provides a counterexample to Haagerup's first problem posed in 1975.
\end{abstract}

\maketitle

\section{Introduction}
In 1975, Uffe Haagerup showed in his master's thesis that several conditions for normality of weights on a von Neumann algebra are all equivalent, and posed three problems on normal weights, numbered 1.10, 1.11, and 2.7 in his paper \cite{MR380438}.
The last problem 2.7 was recently resolved affirmatively by the first author in \cite{MR4972190}, while Bikchentaev investigated the first problem 1.10 in \cite{MR2918421} and \cite{MR3230397}, obtaining partial results in the commutative and finite-dimensional cases.
The second problem 1.11 seems to be still open.

This paper proves that the answer to the first problem is negative.
For a subadditive weight $\varphi$ on a von Neumann algebra $M$, we say $\varphi$ \emph{preserves directed suprema} if $\varphi(\sup_ix_i)=\sup_i\varphi(x_i)$ for any bounded non-decreasing net $x_i\in M$ of positive elements, and $\varphi$ is \emph{$\sigma$-weakly lower semi-continuous} if $\varphi(\lim_ix_i)\le\liminf_i\varphi(x_i)$ for any $\sigma$-weakly convergent net $x_i\in M$ of positive elements.
It is easy to see that every $\sigma$-weakly lower semi-continuous subadditive weight preserves directed suprema, and the first problem of Haagerup asks if the converse is also true.

We use the notation $\mathbb{N}:=\mathbb{Z}_{\ge1}$, and the positive cone of a ordered vector space $E$ will be denoted by $E^+$.

\section*{AI disclosure statement}
The counterexample in this paper was constructed by the second author using ChatGPT-5.6 Sol, and the proof was checked and rewritten by the first author.
The authors take full responsibility for the content and correctness of this paper.

\section{A counterexample}

\begin{theorem*}
There is a subadditive weight on the von Neumann algebra $B(H)$ of bounded linear operators on a separable infinite-dimensional Hilbert space $H$ that preserves directed suprema but is not $\sigma$-weakly lower semi-continuous.
\end{theorem*}
\begin{proof}
We identify $H=\mathbb{C}\varepsilon_0\oplus\ell^2(\mathbb{N})$ with the standard orthonormal basis $\varepsilon_k\in H$ indexed by $k\ge0$.
For $n\ge1$, let $q_n\in B(H)^+$ be such that $q_n/2$ is the rank-one projection onto $\mathbb{C}(\varepsilon_0+\varepsilon_n)$ so that $\langle q_n\varepsilon_i,\varepsilon_j\rangle=1$ if and only if $\{i,j\}\subset\{0,n\}$, and introduce a positive linear map $\ell^1(\mathbb{N})\to B(H):a\mapsto q_a$ by the norm-convergent series
\[q_a:=\sum_{k=1}^\infty a_kq_k,\qquad a=(a_k)_{k\ge1}\in\ell^1(\mathbb{N}).\]
Define $\varphi:B(H)^+\to[0,\infty]$ by
\[\varphi(x):=\inf\{\|a\|_1:x\le q_a,\ a\in\ell^1(\mathbb{N})^+\},\qquad x\in B(H)^+.\]
By construction, $\varphi$ is a subadditive weight.
We will show that it preserves directed suprema, but is not $\sigma$-weakly lower semi-continuous.

Let $p\in B(H)$ be the projection onto $\ell^2(\mathbb{N})$, and let $\eta_n:=\varepsilon_0-\sum_{k=1}^n\varepsilon_k$ for $n\ge1$.
Then for any $a\in\ell^1(\mathbb{N})^+$, the compression $pq_ap=\mathrm{diag}(a_1,a_2,\cdots)\in B(pH)$ of $q_a$ on $pH$ is a diagonal operator, and we have
\[\langle q_a\eta_n,\eta_n\rangle=\sum_{k=n+1}^\infty a_k\to0,\qquad n\to\infty.\]

First we claim that $\varphi$ preserves directed suprema.
Let $x_i\in B(H)^+$ be a bounded non-decreasing net with the supremum $x\in B(H)^+$ with $r:=\sup_i\varphi(x_i)$.
Since the case $r=\infty$ is clear, assume that $r<\infty$.
Fix $\varepsilon>0$ and choose $a_i\in\ell^1(\mathbb{N})^+$ for each $i$ such that $x_i\le q_{a_i}$ and $\|a_i\|_1\le r+\varepsilon$.
Take a subnet $a_j$ of $a_i$ such that $a_j\to a$ weakly$^*$ in $\ell^1(\mathbb{N})\cong c_0(\mathbb{N})^*$.
Because each $\sigma$-weakly continuous linear functional on the subalgebra $\ell^\infty(\mathbb{N})\subset B(pH)$ of diagonal operators sits in $\ell^1(\mathbb{N})\subset c_0(\mathbb{N})$, we have $\sigma$-weak convergence $pq_{a_j}p\to pq_ap$ in $B(pH)$, so we obtain $px_ip\le pq_ap$ for all $i$ by the cofinality of the subnet.
Let $\xi\in H$ be arbitrary such that $\langle\xi,\varepsilon_0\rangle=1$, and set $\xi_n:=\xi-\eta_n$ for $n\ge1$.
For each fixed $i$, since $\|x_i^{\frac12}\eta_n\|^2=\langle x_i\eta_n,\eta_n\rangle\le\langle q_{a_i}\eta_n,\eta_n\rangle\to0$ we have $\langle x_i\xi_n,\xi_n\rangle\to\langle x_i\xi,\xi\rangle$ as $n\to\infty$, and $\xi_n\in pH$ shows $0\le\langle p(q_a-x_i)p\xi_n,\xi_n\rangle=\langle(q_a-x_i)\xi_n,\xi_n\rangle\to\langle(q_a-x_i)\xi,\xi\rangle$.
By limiting $i$ and taking $\xi$ arbitrary we have $x\le q_a$, and it implies $\varphi(x)\le\|a\|_1\le r+\varepsilon$.
Hence the claim follows by $\varepsilon\to0$.

Now we prove $\varphi$ is not $\sigma$-weakly lower semi-continuous.
Observe that $q_n\to1-p$ $\sigma$-weakly in $B(H)$ as $n\to\infty$.
We can see easily $\varphi(q_n)\le1$ by definition of $\varphi$, but $\langle(1-p)\eta_n,\eta_n\rangle=1\not\to0$ as $n\to\infty$ implies that there is no $a\in\ell^1(\mathbb{N})^+$ such that $1-p\le q_a$, meaning $\varphi(1-p)=\infty$.
This completes the proof.
\end{proof}

\bibliographystyle{alpha}
\bibliography{bib}

\end{document}